\documentclass{amsart}
\usepackage{mathrsfs,amssymb,multirow,setspace,color}
\usepackage{hyperref}
\usepackage[a4paper, total={7in, 9in}]{geometry}
\theoremstyle{plain}
\usepackage{graphicx}
\newtheorem{theorem}{Theorem}[section]
\newtheorem{lemma}[theorem]{Lemma}
\newtheorem{proposition}[theorem]{Proposition}
\newtheorem{corollary}[theorem]{Corollary}

\theoremstyle{definition}

\theoremstyle{remark}

\input xy
\xyoption{all}

\def\gh{\Gamma _H(G)}
\def\pgh{\mathcal{P}(\frac{G}{H})}
\def\fgh{\frac{G}{H}}
\begin{document}
\title{Genus and crosscap of Normal subgroup based power graph of finite groups}


\author[Parveen, Manisha, Jitender Kumar]{Parveen, Manisha, Jitender Kumar$^{*}$}
\address{$\text{}^1$Department of Mathematics, Birla Institute of Technology and Science Pilani, Pilani-333031, India}
\email{p.parveenkumar144@gmail.com,yadavmanisha2611@gmail.com,jitenderarora09@gmail.com}

\begin{abstract}
Let $H$ be a normal subgroup of a group $G$. The normal subgroup based  power graph $\Gamma_H(G)$ of $G$ is the simple undirected graph with vertex set $V(\Gamma_H(G))= (G\setminus H)\cup \{e\}$ and two distinct vertices $a$ and $b$ are adjacent if either $aH = b^m H$ or $bH=a^nH$ for some $m,n \in \mathbb{N}$. Mohammadi and Ashrafi (Normal subgroup based power graph of finite groups, \textit{Ricerche di Matematica}, 71(02), 549–559, 2022) studied the normal subgroup based power graph of a finite group. In this paper, we continue the study of normal subgroup based power graph and characterize all the pairs $(G,H)$, where $H$ is a non-trivial normal subgroup of $G$, such that the genus of $\Gamma_H(G)$ is at most $2$. Moreover, we determine all the subgroups $H$ and the quotient groups $\frac{G}{H}$ such that the cross-cap of $\Gamma_H(G)$ is at most three.

 \end{abstract}


\keywords{power graph, genus and cross-cap of a graph. \\ *  Corresponding author}

\maketitle
\section{Historical Background and Main results}

 Graphs associated to finite groups have been studied extensively in the literature. In particular, Cayley graphs, commuting graphs and the power graphs have received considerable attention from various researchers due to their valuable applications (see \cite{a.hayat2019novel,kelarev2003graph,a.kelarev2009cayley,kelarev2004labelled}).
 The \emph{power graph} $\mathcal{P}(G)$ of a finite group $G$ is a simple undirected graph with vertex set $G$ and two vertices $a$, $b$ are adjacent if one of them is a power of the other. 
 Topological graph theory mainly focuses on the representation of graphs on surfaces so that no two edges cross each other. Its applications lies in electronic circuit design, where the main purpose is to embed a circuit (represented as a graph) on a circuit board (the surface) without any two connections crossing each other, which could cause a short circuit. A graph is said to be \emph{embeddable} on a topological surface if it can be drawn on a surface so that no two edges cross each other.
 The \emph{genus} $\gamma(\Gamma)$ of a graph $\Gamma$ is the minimum non-negative integer $g$ such that the graph can be embedded in a sphere with $g$ handles. The problem of calculating the genus of a graph is NP-hard (cf. \cite{a.Thomassen1989}). The \emph{cross-cap} $\overline{\gamma}(\Gamma)$ of a graph $\Gamma$, is the minimum non-negative integer $k$ such that $\Gamma$ can be embedded in $\mathbb{N}_k$, where $\mathbb{N}_k$ is a nonorientable surface with $k$ crosscaps.  Mirzargar \emph{et al.}  \cite{a.Mirzargar2012},  classified all the finite groups whose power graphs are planar. In \cite{a.Doostabadi2017embedding}, the authors characterized the finite groups whose power graphs are of (non)orientable genus one. All the finite groups with (non)orientable genus two power graphs have been characterized in  \cite{a.ma2019power}. The power graphs of finite groups have been studied through various aspects. For more detail on the power graph, we refer the reader to the survey paper \cite{a.kumar2021} and references therein.

Bhuniya and Bera \cite{a.bhuniya2017normal} introduced the normal subgroup based power graph. Let $G$ be a finite group and let $H$ be a normal subgroup of $G$.  The \emph{normal subgroup based power graph} $\Gamma _H(G)$ of $G$ based on the
normal subgroup $H$ is a simple undirected graph with vertex set is $(G\setminus H)\cup \{e\}$ and two distinct vertices $x$ and $y$ are adjacent if and only if $x H = y^m H $ or $y H = x^n H$, for some integers $m$ and $n$. It is a generalization of the usual  power graph in which $H=\{e\}$. They classified all the pairs of group $(G,H)$ of a finite group $G$ and a non-trivial normal subgroup $H$ such that $\gh$ is complete, bipartite, planar, Eulerian and perfect, respectively. Recently, Mohammadi and  Ashrafi \cite{a.malek2022normla} classified of all the pairs $(G,H)$ of a finite group $G$ and a non-trivial proper normal subgroup $H$ of $G$ such that $\gh$ is a split, bisplit and $(n-1)$-bisplit. They proved that the graph $\gh$ is bisplit if and only if $|H|=2$ and $\fgh \cong \mathbb{Z}_2\times \mathbb{Z}_2 \times \cdots \times \mathbb{Z}_2$. Further, they characterized all the pairs $(G,H)$ such that $\gh$ is unicyclic, bicyclic, tricyclic and tetracyclic, respectively. They showed that the graph $\gh$ is unicyclic if and only if $|H|=2$ and $G \cong \mathbb{Z}_4$ or $\mathbb{Z}_2\times \mathbb{Z}_2$. The graph $\gh$ is tricyclic if and only if $|H|= 2$ and $\fgh \cong \mathbb{Z}_2\times \mathbb{Z}_2$ or $|H|= 3$ and $\fgh \cong \mathbb{Z}_2$.

 
Motivated by the above mentioned work on normal subgroup based power graph of a finite group, in this paper, we classify all the pairs $(G,H)$ of a finite group $G$ and a non-trivial proper normal subgroup $H$ of $G$ such that the graph $\gh$ is of genus at most $2$. Additionally, we determine all the subgroups $H$ and the quotient groups $\frac{G}{H}$ such that the cross-cap of $\gh$ is at most three. The main results of this paper are as follows.

\newpage
\begin{theorem}{\label{genus}}
    Let $G$ be a finite group and let $H$ be a proper non-trivial normal subgroup of $G$. Then
    \begin{itemize}
        \item[(i)] $\gamma(\gh)=1$ if and only if one of the following holds: 
        \begin{itemize}
            \item[(a)] $|H|\in \{4,5,6\}$ and $\fgh \cong \mathbb{Z}_2$.
            \item[(b)]  $|H|=3 $ and $\fgh \cong \mathbb{Z}_3$ or $\fgh \cong S_3$.
            \item[(c)]  $|H|=2 $ and $\fgh$ is isomorphic to one of the groups : $\mathbb{Z}_3$, $ S_3$, $ D_8$, $\mathbb{Z}_4$.
        \end{itemize}
         \item[(ii)] $\gamma(\gh)=2$ if and only if one of the following holds: 
        \begin{itemize}
            \item[(a)] $|H| = 7$ and $\fgh \cong \mathbb{Z}_2$.
             \item[(b)]  $|H|=2 $ and $\fgh$ is a $2$-group with exponent $4$ such that $\fgh$ has exactly two cyclic subgroups of order $4$.
        \end{itemize}
    \end{itemize}
\end{theorem}
\begin{theorem}{\label{cross-cap}}
      Let $G$ be a finite group and let $H$ be a proper non-trivial normal subgroup of $G$. Then
      \begin{itemize}
          \item[(i)] $\overline{\gamma}(\gh)=1$ if and only if one of the following holds: 
         \begin{itemize}
            \item[(a)]  $|H|\in \{4,5\}$ and $\fgh \cong \mathbb{Z}_2$.
            \item[(b)]  $|H|=2 $ and $\fgh \cong \mathbb{Z}_3$ or $\fgh \cong S_3$.
        \end{itemize}
          \item[(ii)] the graph $\gh$ cannot have cross-cap $2$. 
          \item[(iii)] $\overline{\gamma}(\gh)=3$ if and only if one of the following holds: 
           \begin{itemize}
             \item[(a)] $|H|= 6 $ and $\fgh \cong \mathbb{Z}_2$.
             \item[(b)] $|H|\in \{4,5\}$ and $\fgh \cong \mathbb{Z}_2\times \mathbb{Z}_2$.
             \item[(c)] $|H|= 3 $ and $\fgh \cong \mathbb{Z}_3$ or $\fgh \cong S_3$.
             \item[(d)] $|H|= 2 $ and $\fgh \cong \mathbb{Z}_4$ or $\fgh \cong D_8$.
           \end{itemize}
     \end{itemize}
\end{theorem}

\section{Preliminaries}
In this section, we recall essential definitions and results. 
Let $G$ be a group with the identity element $e$. The order $o(x)$ of an element $x\in G$ is the smallest positive integer $n$ satisfying  $x^n=e$. We write $\pi _G=\{o(g): g \in G\}$. The least common multiple of the orders of all the elements of $G$  is called the \emph{exponent} of $G$ and it is denoted by $\mathrm{exp}(G)$.  For $d\in \pi _G$, by $s_d(G)$ we mean the number of cyclic subgroups of order $d$ in $G$. A finite group $G$ is called a \emph{p-group} if $|G|=p^{\alpha}$ for some prime $p$ and integer $\alpha$.  By $\langle x, y\rangle$, we mean the subgroup of $G$ generated by $x$ and $y$. Let $H, K $ be subgroups of a group $G$. The subgroup $[H, K] $ of $ G$ is defined as the subgroup generated by all elements of the form $[h, k]:=h^{-1} k^{-1} h k$, where $h \in H, k \in K$. The \emph{lower central series} of $G$ is the descending sequence
$$
G^0 \geq G^1 \geq G^{2} \geq \cdots \geq G^{i} \geq G^{i+1} \geq \cdots
$$
of normal subgroups of $G$, where $G^0 = G$,  $G^1 = [G, G]$ and $G^{i+1} = \left[G^{i}, G\right]$ for $i \in \mathbb{N}$.
Let $G$ be a finite $p$-group of order $p^{n}$. Then $G$ is said to be of \emph{maximal class} if $ G^{n-1} \neq\left\{e\right\}$ and $ G^{n}=\left\{e\right\}$.
The following results are useful in the sequel.

\begin{theorem}{\rm \cite{a.pgroupberkovi,b.pgroupisaac2006,a.pgroupkulakoff,a.pgroupmiller}}{\label{pgroupclass}}
Let $G$ be a finite $p$-group of exponent $p^{k}$. Assume that $G$ is not cyclic for an odd prime $p$, and for $p=2$ it is neither cyclic nor of maximal class. Then
\begin{enumerate}
    \item[(i)] $s_p(G) \equiv 1+p\left(\bmod \ p^{2}\right)$.
    \item[(ii)] $p \mid s_{p^i}(G)$ for every $2 \leq i \leq k$.
\end{enumerate}
\end{theorem}
\begin{corollary}{\rm \cite{a.mishra2021lambda}}{\label{pgroupclasscorollary}}
Let $G$ be a finite $p$-group of exponent $p^{k}$. Then $s_{p^{i}}(G)=1$, for some $1 \leq i \leq k$, if and only if one of the following occurs:
\begin{enumerate}
    \item $G \cong \mathbb{Z}_{p^{k}}$ and $s_{p^{j}}(G)=1$ for all $1 \leq j \leq k$, or
    \item $p=2$ and $G$ is isomorphic to one of the following $2$-groups:
\end{enumerate}
\begin{enumerate}
    \item[(i)] dihedral $2$-group
$$
D_{2^{k+1}}=\left\langle x, y: x^{2^{k}}=1, y^{2}=1, y^{-1} x y=x^{-1}\right\rangle, \quad(k \geq 1)
$$
where $s_2(G)=1+2^{k}$ and $s_{2^{j}}(G)=1 \text{ for all } (2 \leq j \leq k)$.
\item[(ii)] generalized quaternion $2$-group
$$
Q_{2^{k+1}}=\left\langle x, y: x^{2^{k}}=1, x^{2^{k-1}}=y^{2}, y^{-1} x y=x^{-1}\right\rangle, \quad(k \geq 2)
$$
where $s_4(G)=1+2^{k-1}$ and $s_{2^{j}}(G)=1$ for all $1 \leq j \leq k$ and $j \neq 2$.
\item[(iii)] semi-dihedral $2$-group
$$
SD_{2^{k+1}}=\left\langle x, y: x^{2^{k}}=1, y^{2}=1, y^{-1} x y=x^{-1+2^{k-1}}\right\rangle, \quad(k \geq 3)
$$
where $s_2(G)=1+2^{k-1}, \; s_4(G)=1+2^{k-2}$ and $s_{2^{j}}(G)=1$ for all $3 \leq j \leq k$.
\end{enumerate}
\end{corollary}

In view of Corollary \ref{pgroupclasscorollary}, we have the following lemma.
\begin{lemma}{\label{lemma 4}}
    Let $G$ be a finite $p$-group with exponent $p^2$ and $G$ contains exactly one cyclic subgroup of order $p^2$. Then the following holds:
    \begin{itemize}
        \item[(i)] If $p=2$, then $G$ is isomorphic to $\mathbb{Z}_4$ or $D_8$.
        \item[(ii)] If $p>2$, then $G$ is isomorphic to $\mathbb{Z}_{p^2}$.
    \end{itemize}
\end{lemma}
\begin{lemma}{\rm \cite[Lemma 2.1]{a.parveen2023nilpotent}}{\label{intersection lemma}}
Let $G$ be a finite $p$-group with exponent $p^2$. Then either $G$ has exactly one cyclic subgroup of order $p^2$ or $G$ contains at least two cyclic subgroups $M$ and $N$ of order $p^2$ such that $|M\cap N|=p$.
\end{lemma}
\begin{theorem}{\rm \cite[Section 4, I]{b.frobenius1895verallgemeinerung}}{\label{number of cyclic subgroup p}}
    Let $p$ be a prime dividing the order of a group $G$. Then $s_p(G)\equiv 1(\mathrm{mod} \ p)$, where $s_p(G)$ is the number of subgroups of order $p$.
\end{theorem}

\begin{lemma}{\label{uniquesubgroup3}}
    Let $G$ be a finite group such that $\pi_{G}\subseteq \{1,2,3, 4\}$ and let $G$ has a unique cyclic subgroup of order $3$. Then either $G\cong \mathbb{Z}_3$ or $G\cong S_3$.
\end{lemma}

\begin{proof}
    Let $x\in G$ such that $\langle x\rangle$ be the unique cyclic subgroup of order $3$. Since $\pi_{G}\subseteq \{1,2,3,4\}$, we obtain $C_G(x)= \langle x \rangle $. Also, note that for $g\in G$, we have $g^{-1}\langle x\rangle g = \langle x\rangle$. It follows that $\langle x\rangle $ is a normal subgroup of $G$. Consequently, $\frac{G}{\langle x \rangle}$ is isomorphic to a subgroup of $\mathrm{Aut}(\langle x \rangle)$. It implies that $o(G)\in \{3,6\}$. Thus, either $G\cong \mathbb{Z}_3$ or $G\cong S_3$.
\end{proof}

Now we recall the necessary graph theoretic definitions and notions from  \cite{b.godsil2001algebraic,b.westgraph}. A \emph{graph} $\Gamma$ is a structure $(V(\Gamma), E(\Gamma))$, where $V(\Gamma)$ is the vertex set of $\Gamma$ and $E(\Gamma)\subseteq V(\Gamma)\times V(\Gamma)$ is the edge set of $\Gamma$. If $\{u_1,u_2\}\in E(\Gamma)$, then we say that $u_1$ is \emph{adjacent} to $u_2$ and we denote it by $u_1 \sim u_2$. Otherwise, we write it as $u_1 \nsim u_2$. An edge $\{u,v\}$ in a graph $\Gamma$ is said to be a \emph{loop} if $u=v$. A graph $\Gamma$ is called a \emph{simple} graph if it has no loops or multiple edges. Throughout the paper, we consider only simple graphs. A \emph{subgraph} $\Gamma'$ of a graph $\Gamma$ is defined as a graph where $V(\Gamma')$ and $E(\Gamma')$ are subsets of  $V(\Gamma)$ and $E(\Gamma)$, respectively. A subgraph $\Gamma'$ of a graph $\Gamma$ is an \emph{induced subgraph by a set }$X\subseteq V(\Gamma)$ if $V(\Gamma')=X$ and two vertices $u$ and $v$ of $V(\Gamma')$ are adjacent in $\Gamma'$ if and only if they are adjacent in $\Gamma$. A graph $\Gamma$ is called \emph{complete} if any two vertices of $\Gamma$ are adjacent. The complete graph on $n$ vertices is denoted by $K_n$. A graph $\Gamma$ is called \emph{bipartite} if  $V(\Gamma)$ can be partitioned into two subsets such that no two vertices in the same subset are adjacent.  A \emph{complete bipartite} graph is a bipartite graph having its parts sizes $n_1$ and $n_2$ such that every vertex in one part is adjacent to each vertex of the second part and it is denoted by $K_{n_1,n_2}$. In a graph $\Gamma$, the \emph{subdivision} of an edge $\{u,v\}$ is the operation of replacing $\{u,v\}$ with a path $u\sim w \sim v$ through a new vertex $w$. A  \emph{subdivision} of $\Gamma$ is a graph obtained from $\Gamma$ by successive edge subdivisions. Let $\Gamma_1,\ldots , \Gamma_m$ be $m$ graphs such that $V(\Gamma_i)\cap V(\Gamma_j)= \varnothing$, for distinct $i, j$. Then $\Gamma =\Gamma_1 \cup \cdots \cup \Gamma_m$ is a graph with vertex set  $V(\Gamma_1) \cup \cdots \cup V(\Gamma_m)$ and edge set $E(\Gamma_1) \cup \cdots \cup E(\Gamma_m)$. We denote by $mK_n$ the union of $m$ copies of $K_n$. Let $\Gamma_1$ and $\Gamma _2$ be two graphs with disjoint vertex set, the \emph{join} $\Gamma_1 \vee \Gamma_2$ of $\Gamma_1$ and $\Gamma_2$ is the graph obtained from the union of $\Gamma_1$ and $\Gamma_2$ by adding new edges from each vertex of $\Gamma_1$ to every vertex of $\Gamma_2$.  Two graphs $\Gamma$ and $\Gamma'$ are called \emph{isomorphic graphs} if there is a bijection $f$ from $V(\Gamma)$ to $V(\Gamma')$ such that $u\sim v$ in $\Gamma$ if and only if $f(u)\sim f(v)$ in $\Gamma'$.
   A graph $\Gamma$ is \emph{planar} if it can be drawn on a plane surface such that no two edges cross each other. If $\gamma(\Gamma)$ (or $\overline{\gamma}(\Gamma)$) $=0$, then $\Gamma$ is planar. A graph with genus $1$ is called \emph{toroidal graph} and a graph with cross-cap $1$ is called \emph{projective planar graph}.
The following results are useful for determining the genus and cross-cap of a graph.
 \begin{theorem}{\cite{b.westgraph}}{\label{planarcondition1}}
A graph $\Gamma$ is planar if and only if it does not contain a subdivision of $K_5$ or $K_{3,3}$.
 \end{theorem}
\begin{theorem}{\rm \cite[p. 58, p. 152]{b.White1984}} {\label{genuscondition}} 
The genus and cross-cap of the complete graphs $K_n$ and $K_{m, n}$ are given below:
\begin{itemize}
\item[(i)] $\gamma(K_n)= \left\lceil{\frac{(n-3)(n-4)}{12}}\right\rceil $, $n\geq 3$.
\item[(ii)]$\gamma(K_{m,n})=\left\lceil \frac{(m-2)(n-2)}{4}\right\rceil $, $m,n\geq 2$.
\item[(iii)] $\overline{\gamma}(K_n)= \left\lceil{\frac{(n-3)(n-4)}{6}}\right\rceil $, $n\geq 3$, $n\neq 7$; $\overline{\gamma} (K_7)=3$.
\item[(iv)] $\overline{\gamma}(K_{m,n})=\left\lceil \frac{(m-2)(n-2)}{2}\right\rceil $, $m,n\geq 2$.
\end{itemize}
\end{theorem}

A \emph{block} of a graph $\Gamma$ is a maximal connected subgraph $B$ of $\Gamma$ with respect to the property that $B$ remains connected even if we remove a single vertex from  $B$. The following result provides a method to compute the genus and the cross-cap of a graph by using its blocks.

\begin{theorem} [{\rm \cite[Theorem 1]{a.Battle1962}, \cite[Corollary 3]{a.Stahl1977}}]{\label{block}}
     Let $\Gamma$ be a connected graph with $n$ blocks $B_1, \ldots, B_n$. Then 
     \begin{itemize}
         \item[(i)] $\gamma(\Gamma)=\sum_{i=1}^n \gamma\left(B_i\right)$.
         \item[(ii)]  If $\bar{\gamma}\left(B_i\right)=2 \gamma\left(B_i\right)+1$ for each $i$, then
$$
\bar{\gamma}(\Gamma)=1-n+\sum_{i=1}^n \bar{\gamma}\left(B_i\right) .
$$
Otherwise,
$$
\bar{\gamma}(\Gamma)=2 n-\sum_{i=1}^n \mu\left(B_i\right),
$$
where $\mu\left(B_i\right)=\max \left\{2-2 \gamma\left(B_i\right), 2-\bar{\gamma}\left(B_i\right)\right\}$.
     \end{itemize}
\end{theorem}

\begin{theorem}{\cite{b.White1984}}{\label{m_n_g_formula}}
    Let $\Gamma$ be a simple connected graph with $n$ vertices and $m$ edges, where $n\geq 3$. Then $\gamma(\Gamma)\geq \frac{m}{6} - \frac{n}{2}+1$. Furthermore, equality holds if and only if $\Gamma$ has a triangular embedding.
\end{theorem}

\begin{lemma}{\rm \cite[Lemma 3.1.4]{b.Mohar2001graphs}}{\label{crosscap_formula}}
Let $\phi : \Gamma \rightarrow \mathbb{N}_k$ be a $2$-cell embedding of a connected graph $\Gamma$ to the non-orientable surface $\mathbb{N}_k$. Then $v-e+f=2-k$, where $v$, $e$ and $f$ are number of vertices, edges and faces of $\phi (\Gamma)$ respectively, and $k$ is a cross-cap of $\mathbb{N}_k$.
    
\end{lemma}

\section{Proof of the main results}
In this section, we prove the main results of the manuscript. The following results are useful to prove our main results.
  \begin{proposition}[{\rm \cite[Proposition 2.1]{a.bhuniya2017normal}}]{\label{adjency in aH bH}}
Let $H$ be a proper normal subgroup of a group $G$ and let $a H(\neq H)$ and $b H(\neq H)$ be two distinct cosets of $H$. If $a \sim b$ in $\gh$, then each element of $aH$ is adjacent to every element of $b H$ in $\gh$.
\end{proposition}

 \begin{proposition}[{\rm \cite[Proposition 2.3]{a.bhuniya2017normal}}]{\label{adjancy in P_E(G/H}}
    Let $a$ and $b$ be two distinct vertices of the graph $\gh$. Then $a\sim b$ in $\gh$ if and only if either $a H=b H$ or $aH\sim bH$ in the power graph $\pgh$.
\end{proposition}

\begin{theorem}[{\rm \cite[Theorem 3.1]{a.bhuniya2017normal}}]{\label{completenss}}
    Let $G$ be a finite group and $H$ be a non-trivial normal subgroup of $G$. Then the normal subgroup based  power graph $\gh$ is complete if and only if $\fgh$ is a cyclic $p$-group for some prime $p$. 
\end{theorem}
    
\begin{theorem}[{\rm \cite[Theorem 5.4]{a.bhuniya2017normal}}]{\label{planarity}}
    Let $G$ be a finite group and $H$ be a non-trivial normal subgroup of $G$. Then the normal subgroup based  power graph $\gh$ is planar if and only if $|H|=2$ or $3$ and $\fgh \cong \mathbb{Z}_{2} \times \mathbb{Z}_{2} \times \cdots \times \mathbb{Z}_{2}$.
\end{theorem}

\subsection{Proof of Theorems \ref{genus}-\ref{cross-cap}}
    We prove both of the main results through the following cases.\\
  \noindent{\textbf{Case-1:}}  $|H|= 2$. We discuss this case in the following subcases.

  \textbf{Subcase 1.1:} \emph{$\fgh$ has an element of order greater than $6$}. Consider ${aH} \in \fgh$  such that ${o(aH)>6}$. Then the number of generators of the subgroup $\langle {aH} \rangle$ is greater than or equal to $4$. Let $aH, a^{i}H, a^{j}H$ and $ a^{k}H$ be four generators of $\langle {aH} \rangle$. Then all of them are adjacent to one another in the power graph $\pgh$. By Propositions \ref{adjency in aH bH} and \ref{adjancy in P_E(G/H}, the subgraph of $\gh$ induced by the set ${aH \cup a^{i}H \cup a^{j}H \cup a^{k}H \cup \{e\}}$ is isomorphic to $K_{9}$. Thus, $\gamma(\gh)\geq 3$ and $\overline{\gamma}(\gh)\geq 5$ (see Theorems \ref{genuscondition} and \ref{block}).
  
  \textbf{Subcase 1.2:} \emph{$\fgh$ has an element of order $6$}. Let $aH \in \fgh$ such that $o(aH)=6$. In the power graph $\pgh$, the generators $aH$ and $a^{5}H$ are adjacent to all the other elements of $\langle aH \rangle$. 
  Also, ${a^2H} \sim {a^{4}H }$ in $\pgh$. The subgraph of $\gh$ induced by the set ${aH \cup a^{2}H \cup a^{4}H \cup a^{5}H \cup \{e\}}$ is isomorphic to $K_{9}$. Consequently,  $\gamma(\gh)\geq 3$ and $\overline{\gamma}(\gh)\geq 5$.

  \textbf{Subcase 1.3:} \emph{$\fgh$ has an element of order $5$}. By the similar argument used in \textbf{Subcase 1.1}, the graph $\gh$ has a subgraph isomorphic to $K_{9}$. Thus, $\gamma(\gh)\geq 3$ and $\overline{\gamma}(\gh)\geq 5$.
  
  \textbf{Subcase 1.4} $\pi_{\fgh}\subseteq \{1,2,3,4\}$.  In view of Theorem 2.5, we discuss the following subcases.
  
 \textbf{Subcase 1.4(a)}\emph{ $\frac{G}{H}$ has atleast $4$ cyclic subgroups of order $3$}.
 Let ${a_1H}, a_2H, a_3H$ and $a_4H$ be the generators of four cyclic subgroups of $\fgh$. Then for distinct $i,j \in \{ 1,2,3,4\}$, ${{a_iH} \sim {a^{2}_{i}H}}$ and ${a_iH \nsim a_{j}H}$ in $\pgh$. It follows that the subgraph of $\gh$ induced by the set $( \bigcup\limits_{i=1}^{4}a_{i}H) \cup ( \bigcup\limits_{i=1}^{4}a_{i}^{2}H) \cup \{e\}$ is isomorphic to $ K_{1}\vee 4K_{4} $. Therefore, $\gamma(\gh) \geq \gamma(K_{1} \vee 4K_{4}) = 4$ and $\overline{\gamma}(\gh)\geq 4$.
 
  \textbf{Subcase 1.4(b)} \emph{$\frac{G}{H}$ has precisely one cyclic subgroup of order $3$.} By Lemma  \ref{uniquesubgroup3}, either $\fgh \cong \mathbb{Z}_3$ or $\fgh \cong  S_3$. If $\fgh \cong \mathbb{Z}_3$, then $\fgh$ contains two non-identity elements. Consider $aH\in \fgh$ such that $o(aH)=3$. Then $aH\sim a^2H$ in $\pgh$. Consequently, $\gh =K_5$. It follows that $\gamma(\gh)=1$ and $\overline{\gamma}(\gh)=1$. Let $\fgh \cong  S_3$. Then consider $a_1H$, $a_2H$, $a_3H$, $b_1H$, $b_2H\in \fgh$ such that $o(a_1H)=o(a_2H)=o(a_3H)=2$ and $o(b_1H)=o(b_2H)=3$. Observe that for distinct $i,j \in\{1,2,3\}$, $a_iH\nsim a_jH$ and $b_1H\sim b_2H$ in $\pgh$ . Also, for $i\in \{1,2,3\}$ and $j\in \{1,2\}$, we have $a_iH\nsim b_jH$ in $\pgh$. It follows that $\gh = K_1\vee (K_4\cup 3K_2)$. Consequently, $\gamma(\gh)=1$ and $\overline{\gamma}(\gh)=1$.
  
  \textbf{Subcase 1.5} $\fgh$ \emph{is a $2$-group with exponent $4$}. We have the following subcases.

\textbf{Subcase-1.5(a):} $\fgh$ \emph{has at least three cyclic subgroups of order} $4$. Let $a_1H, a_2H$ and $a_3H$ be generators of three distinct cyclic subgroups of order $4$ in $\fgh$. Then by Propositions \ref{adjency in aH bH} and \ref{adjancy in P_E(G/H}, the subgraph of $\gh$ induced by the set $(\bigcup\limits_{i=1}^3 a_iH)\cup (\bigcup\limits_{j=1}^3 a_j^3H)\cup a_1^2H \cup \{e\}$ contains a subgraph which is isomorphic to $K_1\vee (K_6\cup 2K_4)$. Thus, by Theorem \ref{block}, $\gamma(\gh)\geq 3$ and $\overline{\gamma}(\gh)\geq 4$.

\textbf{Subcase-1.5(b):} $\fgh$ \emph{has exactly two cyclic subgroups of order} $4$. Let $a_1H$ and $a_2H$ be generators of two distinct cyclic subgroups of order $4$ in $\fgh$. By Lemma \ref{intersection lemma}, we have $a_1^2H=a_2^2H$. Let $b_1H, b_2H, \ldots ,b_tH$ be all the other elements of order $2$ in $\fgh$. Note that for $1 \leq i \leq 2$, $1\leq j \leq 3$, we have $a_i^jH\nsim b_kH$ in $\pgh$ where $1\leq k \leq t$. Also, for $i,j\in \{1,3\}$, $a_1^iH\nsim a_2^jH$ in $\pgh$. By Propositions \ref{adjency in aH bH} and \ref{adjancy in P_E(G/H}, we obtain $\gh= K_1\vee ((K_2\vee 2K_4)\cup tK_2)$. Notice that $\gamma(\gh) = \gamma(K_3\vee 2K_4) $ and $\overline{\gamma}(\gh) = \overline{\gamma}(K_3\vee 2K_4)$. Since the graph $K_3\vee 2K_4$ has $11$ vertices and $39$ edges, by Theorem \ref{m_n_g_formula}, we have $\gamma(K_3\vee 2K_4)\geq 2$. 
The genus $2$ drawing of  $K_3\vee 2K_4$ is given in Figure \ref{fig K3_2K_4}. Thus, $\gamma(\gh) = 2$. 
  \begin{figure}[ht]
    \centering
    \includegraphics[scale=.65]{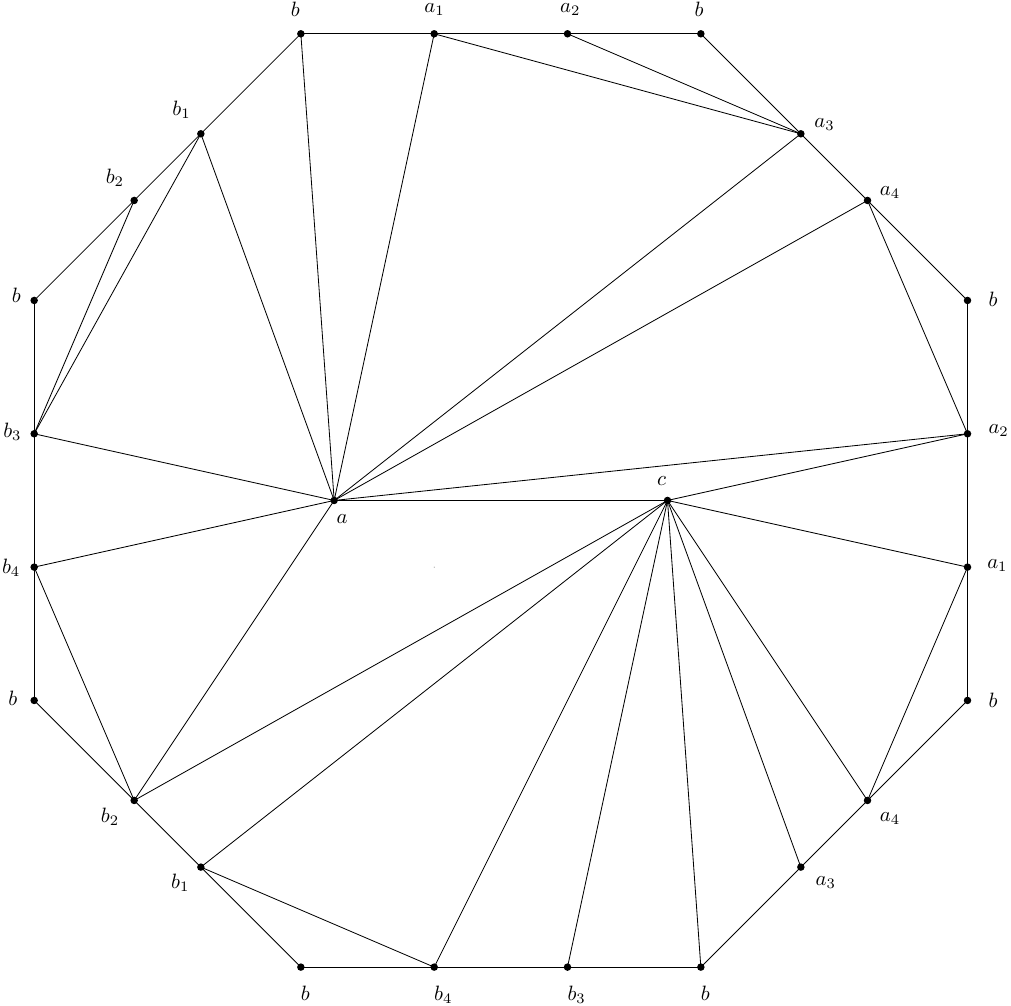}
    \caption{The genus $2$ drawing of $K_3\vee 2K_4$.}
    \label{fig K3_2K_4}
\end{figure}
 Also, $K_3\vee 2K_4$ contains a subgraph which is isomorphic to $K_7$ and so $\overline{\gamma}(\gh) \geq 3$. If $\overline{\gamma}(\gh) = 3$, then by Lemma \ref{crosscap_formula}, the number of faces in $2$-cell embedding of $K_3\vee 2K_4$ in $\mathbb{N}_3$ is $27$. It follows that $2e< 3f$, which is not possible for simple graphs. Consequently, $\overline{\gamma}(\gh) \geq 4$. 
 
 \textbf{Subcase-1.5(c):}  $\fgh$ \emph{has at exactly one cyclic subgroup of order} $4$. By Lemma \ref{lemma 4}, either $\fgh \cong \mathbb{Z}_4$ or $\fgh \cong D_8$. If $\fgh \cong \mathbb{Z}_4$, then by Theorem \ref{completenss}, $\gh\cong K_7$. Thus, $\gamma(\gh) = 1$ and $\overline{\gamma}(\gh) = 3$.  If $\fgh \cong D_8$, then by Propositions \ref{adjency in aH bH} and \ref{adjancy in P_E(G/H}, $\gh = K_1\vee (K_6 \cup 4K_2)$. Note that $\gamma (K_1\vee (K_6 \cup 4K_2)) = \gamma (K_7)$ and $\overline{\gamma} (K_1\vee (K_6 \cup 4K_2)) = \overline{\gamma} (K_7)$. It follows that $\gamma(\gh) = 1$ and $\overline{\gamma}(\gh) = 3$.

\textbf{Subcase-1.6:} $\fgh \cong \mathbb{Z}_2\times \mathbb{Z}_2 \times \cdots \times \mathbb{Z}_2$. By Theorem \ref{planarity}, the graph $\gh$ is planar.\\

\noindent{\textbf{Case-2:}}  $|H|= 3$. We discuss this case into the following subcases.

\textbf{Subcase-2.1:} \emph{$\frac{G}{H}$ has an element of order greater than $4$}. Let $aH \in \fgh$ such that $o(aH)> 4$. Then in $\pgh$, $aH \sim a^{2}H$, $aH \sim a^{4}H $ and $ a^{2}H \sim a^{4}H$ . It follows that the subgraph of $\gh$ induced by the set $aH \cup a^{2}H \cup a^{4}H \cup \{e\}$ is isomorphic to $K_{10}$. By Theorems \ref{genuscondition} and \ref{block}, we have $\gamma(\gh) \geq 4 $ and $\overline{\gamma}(\gh) \geq 7$.


\textbf{Subcase-2.2:} \emph{$\frac{G}{H}$ has an element of order $4$}. Let $aH \in \fgh$ such that $o(aH) = 4$. Then in $\pgh$, $aH \sim a^{2}H$, $aH \sim a^{3}H$ and $ a^{2}H \sim a^{3}H$. Therefore the subgraph of $\gh$ induced by $aH \cup a^{2}H \cup a^{3}H \cup \{e\}$ is isomorphic to $K_{10}$. It implies that  $\gamma(\gh) \geq 4 $ and $\overline{\gamma}(\gh) \geq 7$.

\textbf{Subcase-2.3:} $\pi_{\fgh}\subseteq \{1,2,3\}$ \emph{and $\fgh$ has an element of order $3$}. In view of Theorem \ref{number of cyclic subgroup p}, we have the following subcases.

\textbf{Subcase-2.3(a):} $\fgh$ \emph{has at least four cyclic subgroups of order} $3$. Let $a_1H$, $a_2H$, $a_3H$ and $a_4H$ be generators of four distinct cyclic subgroups of $\fgh$. Then  for $i\in \{1,2,3,4\}$, we have $a_iH\sim a_i^2H$ in $\pgh$. Further, for distinct $j,k \in \{1,2,3,4\}$ and $r,s\in \{1,2\}$, $a_j^rH\nsim a_k^sH$ in $\pgh$. By Propositions \ref{adjency in aH bH} and \ref{adjancy in P_E(G/H}, the subgraph of $\gh$ induced by the set $(\bigcup\limits_{i=1}^{4} a_iH) \cup (\bigcup\limits_{j=1}^{4} a_j^2H) \cup \{e\}$ is isomorphic to $K_1\vee 4K_6$. Consequently, $\gamma(\gh)\geq \gamma(K_1\vee 4K_6)=4$ and $\overline{\gamma}(\gh)\geq 9$ (see Theorems \ref{genuscondition} and \ref{block}).

\textbf{Subcase-2.3(b):} $\fgh$ \emph{has preciously one cyclic subgroup of order} $3$. By Lemma \ref{uniquesubgroup3}, either $\fgh \cong \mathbb{Z}_3$ or $\fgh \cong S_3$. If $\fgh \cong \mathbb{Z}_3$, then $\fgh$ has two non-identity elements. Let $aH$ and $a^2H$ be the non-identity elements of $\fgh$. Then $aH\sim a^2H$ in $\pgh$. Consequently, $\gh = K_7$ (see Propositions \ref{adjency in aH bH} and \ref{adjancy in P_E(G/H}). It follows that $\gamma(\gh)=1 $ and $\overline{\gamma}(\gh)= 3$. Let $\fgh \cong S_3$. Then consider $a_1H$, $a_2H$, $a_3H$, $b_1H$, $b_2H\in \fgh$ such that $o(a_1H)=o(a_2H)=o(a_3H)=2$ and $o(b_1H)=o(b_2H)=3$. Observe that for distinct $i,j \in\{1,2,3\}$, $a_iH\nsim a_jH$ and $b_1H\sim b_2H$ in $\pgh$. Also, for $i\in \{1,2,3\}$ and $j\in \{1,2\}$, $a_iH\nsim b_jH$ in $\pgh$. Thus, $\gh= K_1\vee (K_6\cup 3K_3)$ (see Figure \ref{fig GHS_3}). By Theorems \ref{genuscondition} and \ref{block}, $\gamma(\gh)=1 $. Also, observe that  $\overline{\gamma}(\gh)= \overline{\gamma}(K_7)= 3$.

   \begin{figure}[ht]
    \centering
    \includegraphics[scale=.9]{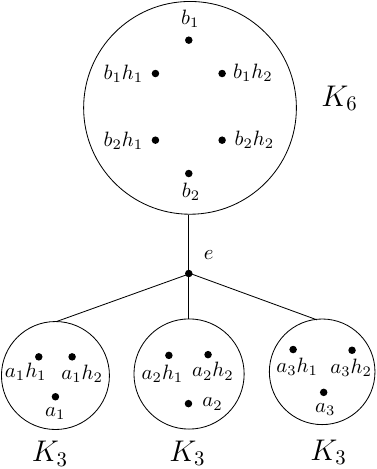}
    \caption{The graph $\gh$, where $\fgh \cong S_3$ and $|H|=3$.}
    \label{fig GHS_3}
\end{figure}

\textbf{Subcase-2.4:} $\fgh \cong \mathbb{Z}_2\times \mathbb{Z}_2 \times \cdots \times \mathbb{Z}_2$. By Theorem \ref{planarity}, the graph $\gh$ is planar.\\

\noindent{\textbf{Case-3:}} $|H|= 4
$. We have the following subcases.

\textbf{Subcase-3.1:} $\fgh$ \emph{contains an element of order greater than $2$}. Let $aH\in \fgh$ such that $o(aH)> 2$. Then $aH\neq a^2H$ and $aH\sim a^2H$ in $\pgh$. The subgraph of $\gh$ induced by the set $aH\cup a^2H \cup \{e\}$ is isomorphic to the complete graph $K_{9}$. Consequently, we obtain  $\gamma(\gh)\geq 3$ and $\overline{\gamma}(\gh)\geq 4$.\\

\textbf{Subcase-3.2:} $\fgh \cong \mathbb{Z}_2\times \mathbb{Z}_2 \times \cdots \times \mathbb{Z}_2$ $k\emph{-times}$, where $k$ $\geq$ 3. Consider the four non-identity elements $a_1H$, $a_2H$, $a_3H$ and $a_4H$ of $\fgh$. Note that for distinct $i,j \in \{1,2,3,4\}$, $a_iH\nsim a_jH$ in $\pgh$. It implies that the subgraph of $\gh$ induced by the set $\bigcup\limits_{i=1}^4 a_iH\cup \{e\}$ is isomorphic to the graph $K_1\vee 4K_4$. Thus, by Theorems \ref{genuscondition} and \ref{block}, $\gamma(\gh)\geq 4\gamma(K_5) = 4$ and $\overline{\gamma}(\gh)\geq 4$.

\textbf{Subcase-3.3:} $\fgh \cong \mathbb{Z}_2\times \mathbb{Z}_2$. Let $a_1H$, $a_2H$ and $a_3H$ be the three non-identity elements of $\fgh$. Observe that for distinct $i,j \in \{1,2,3\}$, we have $a_iH\nsim a_jH$ in $\pgh$. Consequently, $\gh= K_1\vee 3K_4$ and so $\gamma(\gh)=3$ and $\overline{\gamma}(\gh)= 3$.

\textbf{Subcase-3.4:} $\fgh\cong \mathbb{Z}_2$. Let $aH$  be the non-identity element of $\fgh$. Then $V(\gh)= aH \cup \{e\}$. It follows that $\gh = K_5$. Thus, $\gamma(\gh)=1$ and $\overline{\gamma}(\gh)= 1$. \\



\noindent{\textbf{Case-4:}} $|H|= 5$. We divide this case into the following subcases.

\textbf{Subcase-4.1:} $\fgh$ \emph{contains an element of order greater than} $2$. By the similar argument used in \textbf{Subcase 3.1}, the graph $\gh$ contains a subgraph which is isomorphic to the complete graph $K_{11}$. It follows that  $\gamma(\gh)\geq 5$ and $\overline{\gamma}(\gh)\geq 10$.

\textbf{Subcase-4.2:} $\fgh \cong \mathbb{Z}_2\times \mathbb{Z}_2 \times \cdots \times \mathbb{Z}_2$ $k\emph{-times}$, where $k\geq 3$. 
 In the similar lines of \textbf{Subcase 3.2}, we obtain $\gamma(\gh)\geq 4\gamma(K_6)=4$ and $\overline{\gamma}(\gh)\geq 4$.

\textbf{Subcase-4.3:} $\fgh \cong \mathbb{Z}_2\times \mathbb{Z}_2$. Let $a_1H$, $a_2H$ and $a_3H$ be the three non-identity elements of $\fgh$. Observe that $a_iH\nsim a_jH$ in $\pgh$ for distinct $i,j \in \{1,2,3\}$. Consequently, $\gh= K_1\vee 3K_5$ and so $\gamma(\gh)=3$ and $\overline{\gamma}(\gh)= 3$.


\textbf{Subcase-4.4:} $\fgh\cong \mathbb{Z}_2$. Let $aH$  be the non-identity element of $\fgh$. Then $V(\gh)= aH \cup \{e\}$. It follows that $\gh = K_6$. Thus, $\gamma(\gh)=1$ and $\overline{\gamma}(\gh)= 1$. \\

\noindent{\textbf{Case-5:}} $|H|= 6$. We discuss this case through the following subcases.

\textbf{Subcase-5.1:} \emph{$\frac{G}{H}$ contains an element of order greater than $2$}. Let $aH$ be an element of order greater than $2$ in $\frac{G}{H}$. Then $aH \neq a^{2}H$ and $aH \sim a^{2}H$ in $\pgh$. The subgraph of $\gh$ induced by the set $aH \cup a^{2}H \cup \{e\}$ is isomorphic to the complete graph $K_{13}$. Consequently, $\gamma(\gh)\geq 8$ and $\overline{\gamma}(\gh)\geq 15$.

\textbf{Subcase-5.2:} $\fgh \cong \mathbb{Z}_2\times \mathbb{Z}_2 \times \cdots \times \mathbb{Z}_2$  $k\emph{-times}$, where $k\geq 2$. Let $a_{1}H, a_{2}H$ and $a_{3}H$ be three non-identity elements of $\frac{G}{H}$. Note that $a_{i}H \nsim a_{j}H$ in $\pgh$ for distinct $i,j \in \{1,2,3\}$. It implies that the subgraph of $\gh$ induced by the set $a_{1}H \cup a_{2}H \cup a_{3}H \cup \{e\}$ is isomorphic to the graph $K_{1} \vee 3K_{6}$. Thus, by Theorems \ref{genuscondition} and \ref{block}, $\gamma(\gh) \geq 3\gamma(K_{7}) =3$ and $\overline{\gamma}(\gh)\geq 7$.

\textbf{Subcase-5.3:} $\fgh\cong \mathbb{Z}_2$. Let $aH$  be the non-identity element of $\fgh$. Then $V(\gh)= aH \cup \{e\}$. It follows that $\gh = K_7$. Thus, $\gamma(\gh)=1$ and $\overline{\gamma}(\gh)= 3$.\\ 

\noindent{\textbf{Case-6:}} $|H|=7$. This case is elaborated in the subsequent two subcases.

\textbf{Subcase-6.1:} $|\frac{G}{H}| \geq 3$. Let $aH$ and $bH$ be two non-identity elements of $\frac{G}{H}$. If $aH \sim bH$ in $\pgh$, then by Propositions \ref{adjancy in P_E(G/H} and \ref{adjency in aH bH}, the subgraph of $\gh$ induced by the set $aH \cup bH \cup \{e\}$ is isomorphic to the complete graph $K_{15}$. Thus, $\gamma(\gh) \geq 11$ and $\overline{\gamma}(\gh) \geq 22$. Now suppose $aH \nsim bH$ in $\pgh$. By Theorems \ref{genuscondition} and \ref{block}, we obtain $\gamma(\gh) \geq 2\gamma(K_{8})=4$ and $\overline{\gamma}(\gh)\geq 8$.

\textbf{Subcase-6.2:} $\fgh\cong \mathbb{Z}_2$. Let $aH$ be the non-identity element of $\fgh$. Then $V(\gh) = aH \cup \{e\}$ and so $\gh = K_{8}$. Thus $\gamma(\gh) = 2$ and $\overline{\gamma}(\gh) = 4$.\\ 

\noindent{\textbf{Case-7:}} $|H|\geq 8$. Now, let $a \in G\setminus{H}$. Then $aH \neq H$ and $|aH| \geq 8$. Thus the subgraph of $\gh$ induced by the set $aH \cup \{e\}$ is isomorphic to the complete graph $K_{n},$ where $n \geq 9$. Consequently, $\gamma(\gh) \geq 3$ and $\overline{\gamma}(\gh) \geq 5$.

\section*{Declarations}

\textbf{Funding}: The first author gratefully acknowledges for providing financial support to CSIR (09/719(0110)/2019-EMR-I) government of India. The second author and the third author wishes to acknowledge the support of Core Research Grant (CRG/2022/001142) funded by  SERB.

\textbf{Conflicts of interest/Competing interests}: There is no conflict of interest regarding the publishing of this paper. 

\vspace{.3cm}
\textbf{Availability of data and material (data transparency)}: Not applicable.

\vspace{.3cm}
\textbf{Code availability (software application or custom code)}: Not applicable.


\vspace{1cm}
\noindent
{\bf Parveen\textsuperscript{\normalfont 1}, \bf Manisha\textsuperscript{\normalfont 1}, \bf Jitender Kumar\textsuperscript{\normalfont 1}}
\bigskip

\noindent{\bf Addresses}:

\vspace{5pt}

\end{document}